\documentclass[a4paper,11pt]{article}

\usepackage[top=3.0cm, bottom=3.0cm, inner=3.0cm, outer=3.0cm, includefoot]{geometry}

\usepackage{verbatim}

\usepackage{geometry}
\usepackage{amssymb}
\usepackage{amsmath}
\usepackage{graphicx}
\usepackage{amsthm}
\usepackage{url}

\usepackage{color}

\usepackage[utf8]{inputenc}
\usepackage{authblk}

\usepackage{enumerate}
\usepackage{tikz}
\usetikzlibrary{arrows}
\usetikzlibrary{positioning}
\usetikzlibrary{calc}
\usetikzlibrary{decorations.pathmorphing,decorations.markings}
\usetikzlibrary{shapes}
\usetikzlibrary{patterns}

\newcommand{\eChar}{\begin{enumerate}[(i)]}
\newcommand{\eCharR}{\begin{enumerate}[(a)]}
\newcommand{\eBr}{\begin{enumerate}[(1)]}

\title
{The Graph Curvature Calculator and the  curvatures of cubic graphs}

\author[1]{D. Cushing}
\author[2]{R. Kangaslampi}
\author[2]{V. Lipi\"ainen}
\author[3]{S. Liu}
\author[4]{G. W. Stagg}
\affil[1]{Department of Mathematical Sciences, Durham University}
\affil[2]{Department of Mathematics and Systems Analysis, Aalto University }
\affil[3]{School of Mathematical Sciences, University of Science and Technology of China}
\affil[4]{School of Mathematics, Statistics and Physics, Newcastle University }
\date{\today}

\theoremstyle{plain}
\newtheorem{lemma}{Lemma}[section]
\newtheorem{theorem}[lemma]{Theorem}
\newtheorem{proposition}[lemma]{Proposition}
\newtheorem{corollary}[lemma]{Corollary}

\theoremstyle{definition}

\newtheorem{definition}{Definition}[section]

\newtheorem{rem}[lemma]{Remark}
\newtheorem{defn}[lemma]{Definition}

\numberwithin{equation}{section}

\begin{document}

\maketitle

\begin{abstract}
We classify all cubic graphs with either non-negative Ollivier-Ricci curvature or non-negative Bakry-\'Emery curvature everywhere. We show in both curvature notions that the non-negatively curved graphs are the prism graphs and the M\"obius ladders. We also highlight an online tool for calculating the curvature of graphs under several variants of the curvature notions that we use in the classification. As a consequence of the classification result we show that non-negatively curved cubic expanders do not exist.
\end{abstract}

\section{Introduction and statement of results}
Ricci curvature is a fundamental notion in the study of Riemannian manifolds. This notion has been generalised in various ways from the smooth setting of manifolds  to more general metric spaces. This article considers Ricci curvature notions in the discrete setting of graphs. Several adaptations of Ricci curvature such as Bakry-\' Emery curvature (see e.g. \cite{LY10,KKRT16,CLP2016}), Ollivier-Ricci curvature \cite{Oll},  Entropic curvature introduced by Erbar and Maas \cite{EM12}, and Forman curvature \cite{Forman,SKM}, have emerged on graphs in recent years, and there is very active research on these notions. We refer to \cite{NR17} and the references therein for this vibrant research field. 

We focus on the  Bakry-\' Emery curvature and Ollivier-Ricci curvature  in this article. Various modifications of Ollivier-Ricci curvature \cite{LLY11, MF, BCLMP} will also be considered.   Those discrete Ricci curvature notions have also been shown to play significant roles in various applied fields, including:
\begin{itemize}
\item
Studying complex biological networks, such as cancer \cite{SGR},  brain connectivity \cite{Far}, and phylogenetic trees \cite{WM17}.
\item
Quantifying the systemic risk and fragility of financial systems, see \cite{SGR2}.
\item
Investigating node degree, the clustering coefficient and global measures on the internet topology, see \cite{Ni}.

\item Studying the ``congestion'' phenomenon in wireless networks under the heat-diffusion protocol, see \cite{WJBan}.

\item Fast approximating to the tree-width of a graph and applications to determining whether a Quadratic Unconstrained Binary Optimization problem is solvable on the D-Wave quantum computer, see \cite{WJB}.

\item Studying the problem of quantum grativity, see \cite{TrPRE, TrHE}.
\end{itemize}

Because of the complexity of the calculations of these curvature notions on graphs it has proven useful to develop software for dealing with the computations. We will present an online tool that calculates the curvature of graphs in many notions, and use it for some of the calculations of our main theorem:

\begin{theorem}\label{main_thm}
Let $G=(V,E)$ be a cubic graph. Then the following are equivalent: 
\begin{enumerate}
\item[i)]
Each vertex $x\in V$ satisfies the Bakry-\'Emery curvature-dimension inequality $CD(0,\infty)$;
\item[ii)]
Each edge $xy\in E$ has Olliver-Ricci curvature $\kappa_{0}(x,y)\geq 0$;
\item[iii)] $G$ is a prism graph or a M\"obius ladder.
\end{enumerate}
\end{theorem}

An important aspect of applying spectral graph theory to theoretical computer science is the study of the spectral gap of the Laplacian. Expander graphs are highly connected sparse finite graphs. They play important roles in both pure and applied mathematics (see, e.g., \cite{Lubotzky}). The construction of a family of expander graphs, i.e., a family of $d$-regular graphs with increasing sizes and uniformly bounded spectral gaps, is a central topic in this field (see, e.g., \cite{Margulis, LPS, MSS}). 

It has been shown that \emph{positive} lower bounds on curvature ensure the existence of spectral gaps (see \cite{Oll, BJL, BCLL, LLY11, LMP16}). However, a graph with a positive lower bound of curvature can not be arbitrarily large. In fact, it has a bounded diameter (see \cite{Oll, LLY11, LMP17}). Therefore, in particular, there exists no families of expander graphs in the space of positively curved graphs.  
An important open question in this area is on the existence of expander graphs in the space of \emph{non-negatively} curved graphs. This question for Ollivier-Ricci curvature has been asked in \cite[Problem T]{Oll10} (Ollivier mentioned this problem was suggested by A. Noar and E. Milman), and for Bakry-\'Emery curvature in \cite[Question 4.8]{LP14}. We will show under both Bakry-\' Emery and Ollivier-Ricci curvature notions that no non-negatively curved cubic expanders exist.

\section{Definitions}
Throughout this article, let $G=(V,E)$ be a locally finite graph with
vertex set $V$, edge set $E$, and which contains no multiple edges or
self loops. Let $d_x$ denote the degree of the vertex $x\in V$ and
$d(x,y)$ denote the length of the shortest path between two vertices
$x$ and $y$. We denote the existence of an edge between $x$ and $y$ by $x\sim y.$
A graph $G = (V,E)$ is called {\it cubic} if it is $3$-regular, that is, if $d_{v} = 3$ for every $v\in V.$

A {\it prism graph}, denoted $Y_n$, is a graph corresponding to the skeleton of an $n$-prism (see Figure \ref{fig_prism}). Prism graphs are therefore both planar and polyhedral. An $n$-prism graph has $2n$ nodes and $3n$ edges. A {\it M\"obius ladder} $M_n$ is a  graph obtained by introducing a twist in a prism graph $Y_n$, see Figure \ref{fig_mobius}. The M\"obius ladder $M_n$ can also be defined as a cycle $C_{2n}$, where the opposite vertices have been joined together.

\begin{figure}[h!]
  \centering
\begin{tikzpicture}
  \tikzset{vertex/.style={circle, draw, fill=black!50,
                        inner sep=0pt, minimum width=4pt}}
  \tikzset{emptyvertex/.style={shape=circle,minimum size=0.6cm}}

  \node (a1) [vertex]{} ;
  \node (b1) [vertex, below=of a1] {$$};
  \node (a2) [vertex, right=of a1] {};
  \node (b2) [vertex, right=of b1] {};
  \node (a3) [vertex, right=of a2] {};
  \node (b3) [vertex, right=of b2] {};
  \node (a4) [emptyvertex, right=of a3] {};
  \node (b4) [emptyvertex, right=of b3] {};
  \node (an1) [vertex, right=of a4] {};
  \node (bn1) [vertex, right=of b4] {};
  \node (an) [vertex, right=of an1] {};
  \node (bn) [vertex, right=of bn1] {};
  \path
  (a1) edge (b1)
  (a1) edge (a2)
  (b1) edge (b2)
  (a2) edge (b2)
  (a2) edge (a3)
  (b2) edge (b3)
  (a3) edge (b3)
  (a3) edge[dashed] (a4)
  (b3) edge[dashed] (b4)
  (a4) edge[dashed] (an1)
  (b4) edge[dashed] (bn1)
  (an1) edge (bn1)
  (an1) edge (an)
  (bn1) edge (bn)
  (an) edge (bn)
  ;
\draw (a1) to [out=20,in=160, distance=2cm] node[draw=none,fill=none] {} ++(an) ;
\draw (b1) to [out=340,in=200, distance=2cm] node[draw=none,fill=none] {} ++(an) ;
\end{tikzpicture}
  \caption{The prism graph $Y_n$}\label{fig_prism}
\end{figure}
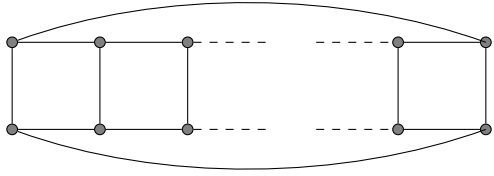

\begin{figure}[h!]
  \centering
\begin{tikzpicture}
  \tikzset{vertex/.style={circle, draw, fill=black!50,
                        inner sep=0pt, minimum width=4pt}}
  \tikzset{emptyvertex/.style={shape=circle,minimum size=0.6cm}}

  \node (a1) [vertex]{} ;
  \node (b1) [vertex, below=of a1] {$$};
  \node (a2) [vertex, right=of a1] {};
  \node (b2) [vertex, right=of b1] {};
  \node (a3) [vertex, right=of a2] {};
  \node (b3) [vertex, right=of b2] {};
  \node (a4) [emptyvertex, right=of a3] {};
  \node (b4) [emptyvertex, right=of b3] {};
  \node (an1) [vertex, right=of a4] {};
  \node (bn1) [vertex, right=of b4] {};
  \node (an) [vertex, right=of an1] {};
  \node (bn) [vertex, right=of bn1] {};
  \path
  (a1) edge (b1)
  (a1) edge (a2)
  (b1) edge (b2)
  (a2) edge (b2)
  (a2) edge (a3)
  (b2) edge (b3)
  (a3) edge (b3)
  (a3) edge[dashed] (a4)
  (b3) edge[dashed] (b4)
  (a4) edge[dashed] (an1)
  (b4) edge[dashed] (bn1)
  (an1) edge (bn1)
  (an1) edge (an)
  (bn1) edge (bn)
  (an) edge (bn)
  (a1) edge (bn)
  (b1) edge (an)
  ;
\end{tikzpicture}
  \caption{The M\"obius ladder $M_n$}\label{fig_mobius}
\end{figure}
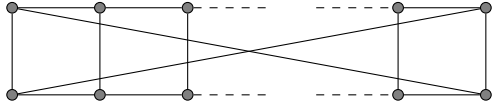

\subsection{Bakry-\'Emery curvature}
 For any function $f: V\to \mathbb{R}$ and any
vertex $x\in V$, the \emph{(non-normalized) Laplacian} $\Delta$ is
defined via
\begin{equation}\label{eq:nonnormalised_Laplacian}
\Delta f(x):=\sum_{y,y\sim x}(f(y)-f(x)).
\end{equation}
The notion of a Laplacian can be generalised by introducing a vertex
measure and edge weights. In this article we will only consider
curvature associated to the non-normalized Laplacian.

\begin{definition}[$\Gamma$ and $\Gamma_{2}$ operators]\label{defn:GammaGamma2}
Let $G=(V,E)$ be a locally finite simple graph.
For any two functions $f,g: V\to \mathbb{R}$, we define
\begin{align*}
2\Gamma(f,g)&:=\Delta(fg)-f\Delta g-g\Delta f;\\
2\Gamma_2(f,g)&:=\Delta\Gamma(f,g)-\Gamma(f,\Delta g)-\Gamma(\Delta f,g).
\end{align*}
\end{definition}
We will write $\Gamma(f):=\Gamma(f,f)$ and $\Gamma_2(f,f):=\Gamma_2(f)$, for short.

\begin{definition}[Bakry-\'Emery curvature]\label{defn:BEcurvature} Let $G=(V,E)$ be a locally finite simple graph. Let $\mathcal{K}\in \mathbb{R}$ and $\mathcal{N}\in (0,\infty]$. We say that a vertex $x\in V$ satisfies the \emph{curvature-dimension inequality} $CD(\mathcal{K},\mathcal{N})$, if for any $f:V\to \mathbb{R}$, we have
\begin{equation}\label{eq:CDineq}
\Gamma_2(f)(x)\geq \frac{1}{\mathcal{N}}(\Delta f(x))^2+\mathcal{K}\Gamma(f)(x).
\end{equation}
We call $\mathcal{K}$ a lower Ricci curvature bound of $x$, and $\mathcal{N}$ a dimension parameter. The graph $G=(V,E)$ satisfies $CD(\mathcal{K},\mathcal{N})$ (globally), if all its vertices satisfy $CD(\mathcal{K},\mathcal{N})$.
\end{definition}

In this paper we only wish to find graphs that satisfy $CD(0,\infty).$ Thus Equation (\ref{eq:CDineq}) becomes
\begin{equation}\label{eq:CDineq2}
\Gamma_2(f)(x)\geq 0.
\end{equation}
We call  such graphs {\it non-negatively curved} under the Bakry-\' Emery curvature notion.

\subsection{Ollivier-Ricci curvature}

In order to define the Ollivier-Ricci curvature on the edges of a graph, we need to define the probability distributions that we consider and the Wasserstein distance between  distributions. So, let us define the following probability distributions $\mu^p_x$ for any
$x\in V,\: p\in[0,1]$:
$$\mu_x^p(z):=\begin{cases}p,&\text{if $z = x$,}\\
\frac{1-p}{d_x},&\text{if $z\sim x$,}\\
0,& \mbox{otherwise.}\end{cases}$$

\begin{defn}
Let $G = (V,E)$ be a locally finite graph. Let $\mu_{1},\mu_{2}$ be two probability measures on $V$. The {\it Wasserstein distance} $W_1(\mu_{1},\mu_{2})$ between $\mu_{1}$ and $\mu_{2}$ is defined as
\begin{equation} \label{eq:W1def}
W_1(\mu_{1},\mu_{2})=\inf_{\pi} \sum_{y\in V}\sum_{x\in V} d(x,y)\pi(x,y),
\end{equation}
where the infimum is taken over all transportation plans $\pi:V\times  V\rightarrow [0,1]$ satisfying
$$\mu_{1}(x)=\sum_{y\in V}\pi(x,y),\:\:\:\mu_{2}(y)=\sum_{x\in V}\pi(x,y).$$
\end{defn}

The transportation plan $\pi$ takes the distribution $\mu_1$ to the distribution $\mu_2$, and $W_1(\mu_1,\mu_2)$ is a measure for the minimal effort
which is required for such a transition. If $\pi$ attains the infimum in \eqref{eq:W1def} we call it an {\it
  optimal transport plan} transporting $\mu_{1}$ to $\mu_{2}$.

\begin{defn}
The $ p-$Ollivier-Ricci curvature on an edge $x\sim y$ in $G=(V,E)$ is
$$\kappa_{ p}(x,y)=1-W_1(\mu^{ p}_x,\mu^{ p}_y),$$
where the parameter $p$ is called the {\it idleness}.
\end{defn}

From the definition of the Wasserstein metric we see, that we get an upper bound for $W_1$ and thus a lower bound for the curvature by choosing some suitable $\pi$. Using the Kantorovich duality (see e.g. \cite[Ch. 5]{Vill09}), a fundamental concept in the optimal transport theory, we can approximate to the opposite direction:

\begin{theorem}[Kantorovich duality]\label{Kantorovich}
Let $G(V,E)$ be a locally finite graph, and let $\mu_{1},\mu_{2}$ be two probability measures on $V$. Then
$$W_1(\mu_{1},\mu_{2})=\sup_{\substack{\phi\: :\: V\rightarrow \mathbb{R}\\ \phi\:\in\: {\rm \emph{$1$-Lip}}}}  \sum_{x\in V}\phi(x)(\mu_{1}(x)-\mu_{2}(x)),$$
where \emph{$1$-Lip} denotes the set of all $1$-Lipschitz functions.
If $\phi \in$ \emph{$1$-Lip} attains the supremum we call it an \emph{optimal Kantorovich potential} transporting $\mu_{1}$ to $\mu_{2}$.
\end{theorem}

In this article we study cubic graphs with curvature $\kappa_0\geq 0$ on all edges, but the Graph Curvature Calculator calculates Olliver-Ricci curvatures with any idleness $p$.

\section{The Graph Curvature Calculator}

The Graph Curvature Calculator, \url{http://www.mas.ncl.ac.uk/graph-curvature/},  is a tool for calculating curvature of graphs under the various curvature notions described in this article. The software provides a powerful, yet easy to use, interface to the Python software graphcurvature.py \cite{graphcurvature_py}. As the interface to the tool is provided over the Web, interested researchers can investigate graph curvature examples without any required knowledge of Python programming. In this chapter we summarise the design and interface to the Graph Curvature Calculator.

\begin{figure}[h]
\includegraphics[width=\textwidth]{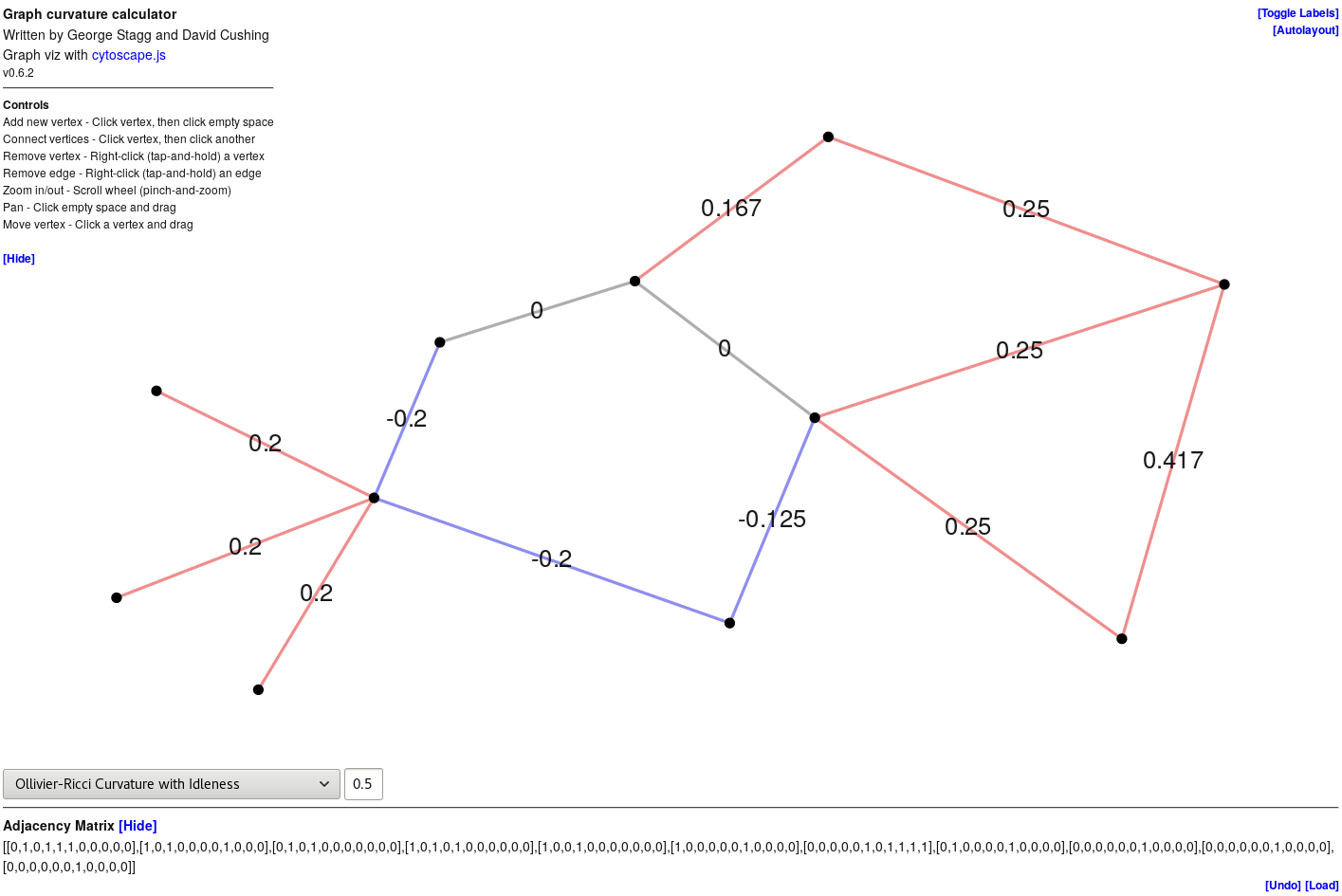}
\caption{The Graph Curvature Calculator shown with a graph loaded and calculating the graph's Ollivier-Ricci Curvature with idleness 0.5.\label{fig:ccscreenshot}}
\end{figure}

\subsection{Architecture}\label{sec:curv_calc_arch}

The calculator is designed around a client-server model, a distributed structure that offloads the computational workload from the user's machine and web browser (the client) to a remote machine or collection of machines (the server). The rationale of this design is that the quality and performance of the user's machine does not need to be particularly high to be able to compute graph curvature. The sophisticated numerical calculations and optimisation problems are instead solved server-side and communicated back to the user's machine. A high level diagram of the architecture of the system is shown in Figure \ref{fig:gcc_arch}.

\begin{figure}[h]
\includegraphics[width=\textwidth]{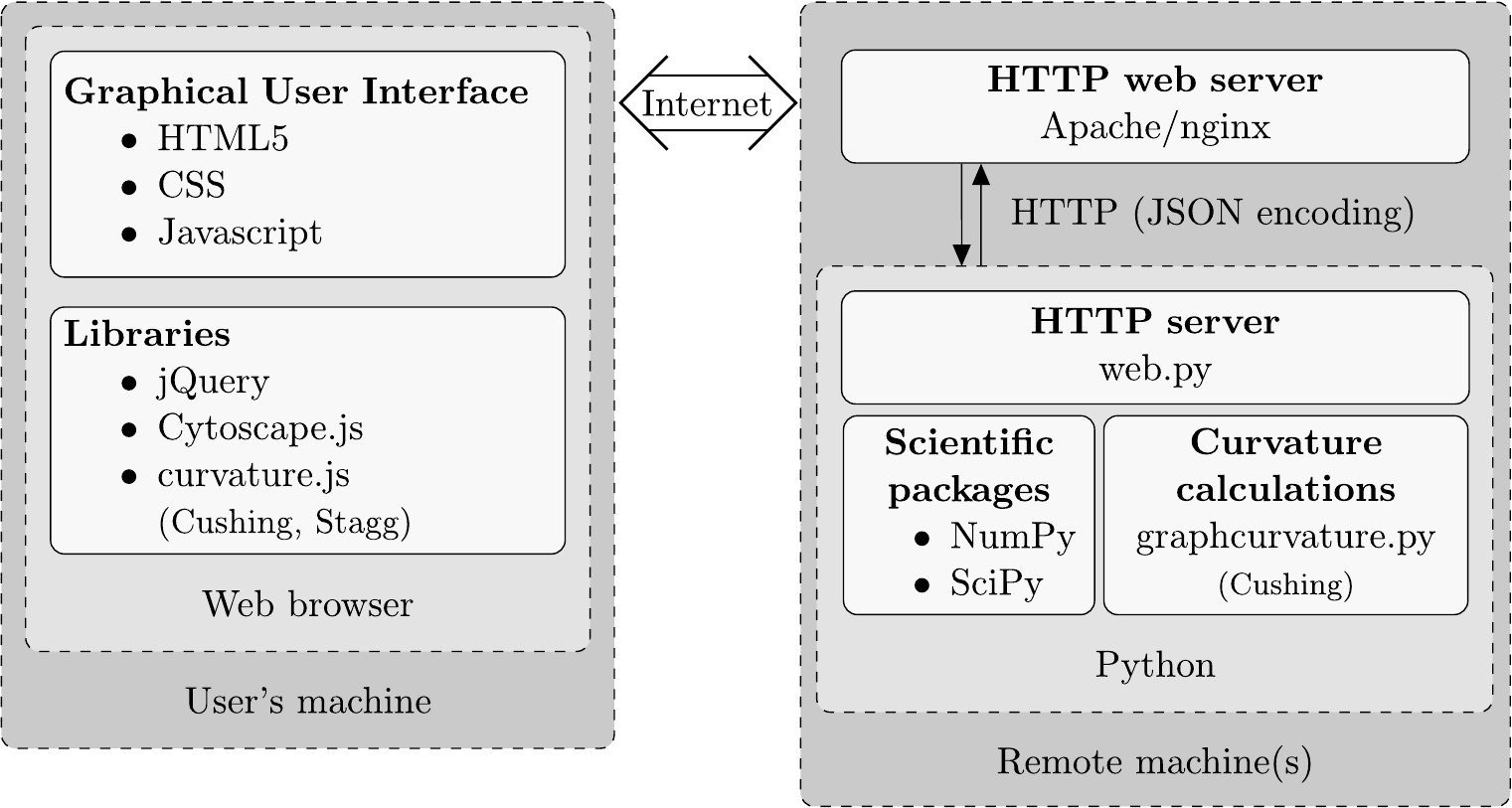}
\caption{Architecture diagram for the Graph Curvature Calculator. The diagram showcases the client-server nature of the configuration, with communication over the Internet using Hypertext Transfer Protocol (HTTP) over TCP connections.\label{fig:gcc_arch}}
\end{figure}

\subsection{The client-side software}
The client-side part of the Graph Curvature Calculator is provided in the form of a \mbox{website}, currently hosted on the public Internet \cite{graph_curv_calc}. The website provides  a Graphical User Interface (GUI) allowing for finite graph input and manipulation, with graph curvature calculations displayed alongside the graph.

The website pages are constructed in HTML and CSS and the client-side software, \mbox{curvature.js}, is built with JavaScript. This is a standard toolchain of web technologies and so should be compatible with all modern web browsers. Cytoscape.js \cite{cytoscape16} provides graph visualisation, and jQuery \cite{jquery} provides additional support. Both JavaScript libraries are free and open-source. While we take full advantage of the graph visualisation routines provided by Cytoscape.js, the curvature calculations are performed on remote machines and so we do not use any of its available analysis routines.

The client-side software allows users to define graphs for curvature calculation in two ways. Firstly, a graph can be loaded into the software by providing an adjacency matrix in so called `JSON' format \cite{json}. For example, the adjacency matrix
\begin{equation*}
\begin{pmatrix}
0&1&1&0\\
1&0&1&1\\
1&1&0&1\\
0&1&1&0
\end{pmatrix}
\end{equation*}
can be inserted by providing the text string \texttt{[[0,1,1,0],[1,0,1,1],[1,1,0,1],[0,1,1,0]]} to the software.

Alternatively, and perhaps more intuitively, the user can ``draw'' graphs by interacting with the website using their mouse or touchscreen device. Inputting graphs by drawing allows for an immediate and constantly updating display of graph curvature. The resulting environment develops the user's intuition for curvature and allows for rapid prototyping of ideas or conjectures.

\subsection{The server-side software}
As the Graph Curvature Calculator is a website, communication between the user's machine and a remote machine is achieved through Hypertext Transfer Protocol (HTTP) over TCP \cite{http}. It is natural, therefore, that the communication of user defined graphs and computational results should also be communicated using HTTP. As such, in our setup the web server also acts as a proxy, passing on requests and responses to and from the computational server where graph curvature is calculated. This choice of setup is extremely flexible with regards to the physical hardware of the remote machines:
\begin{itemize}
\item The computational server need not be on the same physical machine as the web server.
\item While the web server must be visible to the Internet, the computational server need only be visible to the web server rather than the public Internet, improving security.
\item The full state of the graph is transmitted with every curvature request, and so several computational servers can be running simultaneously --- automatically load-balanced by the web server in a `Round-Robin' fashion.
\item Load-balancing can be further improved by running the computational servers in the ``cloud'', bringing virtual machines on-line and off-line to manage system load in real-time.
\end{itemize}

Communication and calculation of graph curvature is achieved using Python on the computational server. The software listens for HTTP requests containing a JSON encoded description of a graph and a requested graph curvature notion. The actual numerical calculation of the graph curvature is performed using the software graphcurvature.py \cite{graphcurvature_py},  with support from the scientific packages NumPy \cite{numpy} and SciPy \cite{scipy}. The curvature calculations are then finally returned to the client-side software as a HTTP response, again encoded in JSON.

\subsection{Bakry-\'Emery curvature as a semidefinite programming problem}
We now reformulate the calculation of Bakry-\'Emery curvature as a semidefinite programming problem. Once this has been achieved it is an easy exercise to be numerically solved.

The following reformulation can also be found in \cite{CLP2016} and \cite{LMP16}.

First we introduce some fundamental notations. For any $r\in\mathbb{N}$,
the $r$-ball centered at $x$ is defined as
$$B_r(x):=\{y\in V: d(x,y)\leq r\},$$
and the $r$-sphere centered at $x$ is
$$S_r(x):=\{y\in V: d(x,y)=r\}.$$
Then we have the following decomposition of the $2$-ball $B_2(x)$:
\begin{equation*}
B_2(x)=\{x\}\sqcup S_1(x)\sqcup S_2(x).
\end{equation*}
We call an edge $\{y,z\}\in E$ a \emph{spherical edge} (w.r.t. $x$) if $d(x,y)= d(x,z)$, and a \emph{radial edge} if otherwise. For a vertex $y\in V$, we define
\begin{align*}
d_y^{x,+}&:=|\{z: z\sim y, d(x,z)>d(x,y)\}|, \\
d_y^{x,0}&:=|\{z: z\sim y, d(x,z)=d(x,y)\}|, \\
d_y^{x,-}&:=|\{z: z\sim y, d(x,z)<d(x,y)\}|.
\end{align*}
In the above, the notation $|\cdot|$ stands for the cardinality of the set. We call $d_y^{x,+}$, $d_y^{x,0}$, and $d_y^{x,-}$ the \emph{out degree}, \emph{spherical degree}, and \emph{in degree} of $y$ w.r.t. $x$. We sometimes write $d_y^+, d_y^0, d_y^-$ for short when the reference vertex $x$ is clear from the context.
\\
\\
We write $\Gamma(x)$ as a $|B_1(x)|\times|B_1(x)|$ matrix corresponding to vertices in $B_1(x)$ given by
\begin{equation*}\label{eq:Gamma}
2\Gamma(x)=\begin{pmatrix}
d_x & -1 & \cdots & -1 \\
-1  & 1 & \cdots & 0 \\
\vdots & \vdots & \ddots &\vdots \\
-1 & 0 & \cdots & 1
\end{pmatrix}.
\end{equation*}

The matrix $\Gamma_2(x)$ is of size $|B_2(x)|\times |B_2(x)|$ with the following structure (\cite[Proposition 3.12]{LMP16})
\begin{equation*}\label{eq:Gamma2}
4\Gamma_2(x)=\begin{pmatrix}
(4\Gamma_2(x))_{x,x} & (4\Gamma_2(x))_{x,S_1(x)} & (4\Gamma_2(x))_{x,S_2(x)} \\
(4\Gamma_2(x))_{S_1(x),x} & (4\Gamma_2(x))_{S_1(x),S_1(x)} & (4\Gamma_2(x))_{S_1(x), S_2(x)}\\
(4\Gamma_2(x))_{S_2(x),x} & (4\Gamma_2(x))_{S_2(x), S_1(x)} & (4\Gamma_2(x))_{S_2(x),S_2(x)}
\end{pmatrix}.
\end{equation*}
The sub-indices indicate the vertices that each submatrix is corresponding to.
We will omit the dependence on $x$ in the above expressions for simplicity. When we exchange the order of the sub-indices, we mean the transpose of the original submatrix. For example, we have $(4\Gamma_2)_{S_1,x}:=((4\Gamma_2)_{x,S_1})^\top$.

Denote the vertices in $S_1(x)$ by $\{y_1,\ldots, y_{d_{x}}\}$. Then we have
\begin{equation*}\label{eq:Gamma2xS1}
(4\Gamma_2)_{x,x}=3d_x+d_x^2,\,\,\,(4\Gamma_2)_{x,S_1}=\begin{pmatrix}
-3-d_x-d_{y_1}^+ & \cdots & -3-d_x-d_{y_{d_x}}^+
\end{pmatrix},
\end{equation*}
and
\begin{align*}
&(4\Gamma_2)_{S_1,S_1}\notag\\
=&\begin{pmatrix}
 5-d_x+3d_{y_1}^++4d^0_{y_1}& 2-4w_{y_1y_2} & \cdots & 2-4w_{y_1y_{d_x}}\\
 2-4w_{y_1y_2} & 5-d_x+3d_{y_2}^++4d^0_{y_2} & \cdots & 2-4w_{y_2y_{d_x}}\\
 \vdots & \vdots & \ddots & \vdots \\
 2-4w_{y_1y_{d_x}} & 2-4w_{y_2y_{d_x}} & \cdots & 5-d_x+3d_{y_{d_x}}^++4d^0_{y_{d_x}}
\end{pmatrix},\label{eq:Gamma2S1S1}
\end{align*}
where we use the notation that for any two vertices $x,y\in V$,
\begin{equation*}\label{eq:0_1_edge_weight}
w_{xy}=\begin{cases}
1, & \text{ if $x\sim y$ }\\
0, & \text{otherwise}.
\end{cases}
\end{equation*}

Denote the vertices in $S_2(x)$ by $\{z_1,\ldots, z_{|S_2(x)|}\}$. Then we have
\begin{equation*}\label{eq:Gamma2xS2}
(4\Gamma_2)_{x,S_2}=\begin{pmatrix}
d_{z_1}^- & d_{z_2}^- & \cdots & d_{z_{|S_2(x)|}}^-
\end{pmatrix}
,
\end{equation*}
\begin{equation*}\label{eq:Gamma2S1S2}
(4\Gamma_2)_{S_1,S_2}=\begin{pmatrix}
-2w_{y_1z_1} & -2w_{y_1z_2} & \cdots & -2w_{y_1z_{|S_2(x)|}}\\
\vdots & \vdots & \ddots & \vdots \\
-2w_{y_{d_x}z_1} & -2w_{y_{d_x}z_2} & \cdots & -2w_{y_{d_x}z_{|S_2(x)|}}
\end{pmatrix}.
\end{equation*}
and
\begin{equation*}\label{eq:Gamma2S2S2}
(4\Gamma_2)_{S_2,S_2}=\begin{pmatrix}
d_{z_1}^- & 0 & \cdots & 0 \\
0 & d_{z_2}^- & \cdots & 0 \\
\vdots & \vdots & \ddots & \vdots\\
0 & 0 & \cdots & d^-_{z_{|S_2(x)|}}
\end{pmatrix}.
\end{equation*}
Note that each diagonal entry of $(4\Gamma_2)_{S_2,S_2}$ is positive.
\\
\\
Let $A(G)$ be the adjacency matrix of the graph $G$. Then we see
$$(4\Gamma_2)_{S_1,S_2}=-2\cdot A(G)_{S_1,S_2}.$$
\\
\\
We are now ready to state the reformulation.
\begin{proposition}[\cite{LMP16}]\label{prop:LMP}
Let $G=(V,E)$ be a locally finite simple graph and let $x\in V$. The Bakry-\'Emery curvature function $\mathcal{K}_{G,x}(\mathcal{N})$ valued at $\mathcal{N}\in (0,\infty]$
 is the solution of the following semidefinite programming,
\begin{align*}
 &\text{maximize}\,\,\, \mathcal{K}\\
&\text{subject to}\,\,\,\Gamma_2(x)-\frac{1}{\mathcal{N}}\Delta(x)^\top\Delta(x)\geq \mathcal{K}\Gamma(x),
\end{align*}
\end{proposition}

This problem was then numerically solved using Python. Some tools from Numpy were used.

The available variations of Bakry-\'Emery curvature on the Graph Curvature Calculator are
\begin{itemize}
\item
Non-normalised curvature sign:
\\
For each vertex $x$ this computes the sign of $\mathcal{K}_{G,x}(\infty)$.
\item
Non-normalised curvature:
\\
For each vertex $x$ this computes the value of $\mathcal{K}_{G,x}(\infty)$ to 3 d.p.
\item
Non-normalised curvature with finite dimension
\\
For each vertex $x$ this computes the value of $\mathcal{K}_{G,x}(\mathcal{N})$ to 3 d.p. for a input dimension $\mathcal{N}.$
\end{itemize}
For each of the above options there are ``Normalised'' analogues in which the above procedure is carried out but with respect to the Normalised Laplacian. See \cite{CLP2016} for further details.

\subsection{Ollivier-Ricci curvature as a linear programming problem}
The problem of calculating Ollivier-Ricci curvature can be reformulated into a linear programming program. We will now give an explanation of how to do this. After this reformulation it is relatively simple to solve numerically. We used the SciPy module in Python for the Graph Curvature Calculator.
\\
\\
Let $G = (V,E)$ be a locally finite simple graph. Let $\mu$ and $\nu$ be two probability measures, with finite supports $\{x_1,x_2,\ldots, x_n\}$ and $\{y_1,y_2,\ldots, y_m\}$ respectively. Recall, by Theorem \ref{Kantorovich}, that
\begin{equation}\label{W1sup}W_1(\mu,\nu)=\sup_{\phi(x_i)-\phi(y_j)\leq d(x_i,y_j)}\left\{\sum_i\phi(x_i)\mu(x_i)-\sum_j\phi(y_j)\nu(y_j)\right\},\end{equation}
where $\phi$ is a function on $\{x_1,x_2,\ldots, x_n\}\bigcup\{y_1,y_2,\ldots, y_m\}$.

We now write (\ref{W1sup}) in the standard form of a linear programming problem.
\\
\\
Let $m,\phi, c, \xi$ be the following column vectors:
\begin{align*}
m:=&(\mu(x_i),\ldots,\mu(x_n),\nu(y_i),\ldots, \nu(y_m))^T\in \mathbb{R}^{n+m},\\
\phi:=&(\phi(x_i),\ldots,\phi(x_n),-\phi(y_i),\ldots,-\phi(y_m))^T\in \mathbb{R}^{n+m},\\
c:=&(d(x_1,y_1),\ldots,d(x_1,y_m),d(x_2,y_1),\ldots,d(x_n,y_1),\ldots, d(x_n,y_m))^T\in \mathbb{R}^{nm}.
\end{align*}

We now define the following $(nm)\times (nm)$ matrix $A.$ First denote by $I_m$ be the $m\times m$ identity matrix and by $a_i$ the $m\times n$ matrix with all the terms in the i-th column equal to $1$, and all terms in the other columns equal to $0$, e.g.
$$a_1=\left(
        \begin{array}{cccc}
          1 & 0 & \cdots & 0 \\
          1 & 0 & \cdots & 0 \\
          \vdots & \vdots & \ddots & \vdots \\
          1 & 0 & \cdots & 0 \\
        \end{array}
      \right).
$$
Then $A$ is defined as
$$A:=\left(
      \begin{array}{cc}
        a_1 & I_m \\
        a_2 & I_m \\
        \vdots & \vdots \\
        a_n & I_m \\
      \end{array}
    \right).
$$

Finally we can rewrite (\ref{W1sup}) as
\begin{equation}\label{LinearPro}
W_1(\mu,\nu)= \sup_{A\phi\leq c}m\cdot\phi.
\end{equation}
This is now just the standard form of the linear programming problem. See \cite{LoRo} for an alternate derivation.

The available variations of Ollivier-Ricci curvature on the Graph Curvature Calculator are
\begin{itemize}
\item
Ollivier-Ricci curvature:
\\
Gives the curvature for idleness $p=0.$
\item
Ollivier-Ricci curvature with idleness:
\\
Gives the curvature for an in-putted idleness $p\in [0,1].$
\item
Lin-Lu-Yau Curvature:
\\
Gives the Lin-Lu-Yau curvature. This is calculated by using
$$\kappa_{LLY}(x,y) = \frac{max(d_{x},d_{y})+1}{max(d_{x},d_{y})}\kappa_{\frac{1}{max(d_{x},d_{y})+1}}(x,y),$$
which is proven in \cite{BCLMP}.
\item
Non-normalised Lin-Lu-Yau curvature. A variant of the Lin-Lu-Yau curvature which is based on a preprint by F. M\"unch and R. Wojciechowski \cite{MF}. 
\end{itemize}

\section{Bakry-\'Emery classification}

Our main result of this section is to classify all cubic graphs satisfying $CD(0,\infty).$ We  start by classifying the local structure of such graphs. It is well known that the curvature at a vertex depends only on the structure of the 2-ball of this vertex, see \cite{CLP2016} for further details. Thus we present the local structure of all possible 2-balls and calculate their curvatures, that is, the lower Ricci curvature bounds $\mathcal{K}$, giving us building blocks for our classification result.

In Figures \ref{fig_A1} to \ref{fig_D7} are screenshots of all possible 2-balls and the curvature of their centres (highlighted in yellow) drawn and calculated using the Graph Curvature Calculator. They are ordered by the number of triangles in the 1-ball around the centre: in the 2-ball $A1$ the centre is on three triangles, in the 2-balls $B1$ and $B2$ on two, in $C1$-$C5$ on one, and in $D1$-$D7$ there are no triangles with the centre vertex.
\begin{figure}
\centering
\includegraphics[scale=0.5]{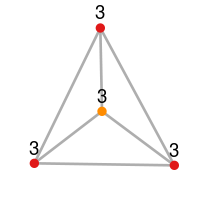}
\includegraphics[scale=0.5]{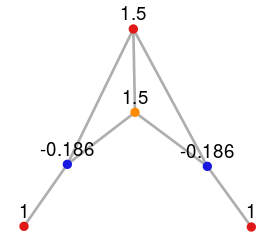}
\includegraphics[scale=0.5]{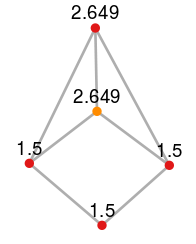}
\caption{The 2-balls $A1$, $B1$ and $B2$}\label{fig_A1}
\end{figure}
\begin{figure}
\centering
\includegraphics[scale=0.5]{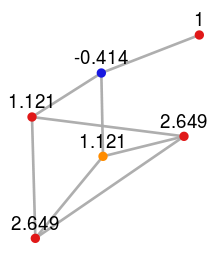}
\includegraphics[scale=0.5]{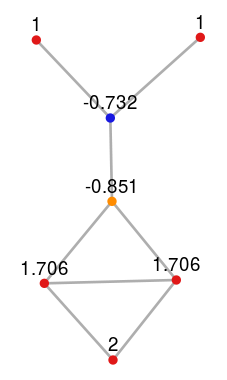}
\includegraphics[scale=0.5]{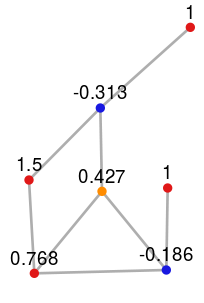}
\caption{The 2-balls $C1$, $C2$ and $C3$}
\end{figure}
\begin{figure}
\centering
\includegraphics[scale=0.5]{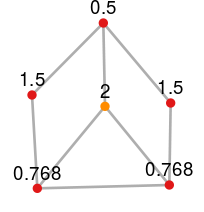}
\includegraphics[scale=0.5]{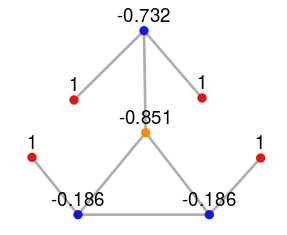}
\caption{The 2-balls $C4$ and $C5$}
\end{figure}
\begin{figure}
\centering
\includegraphics[scale=0.5]{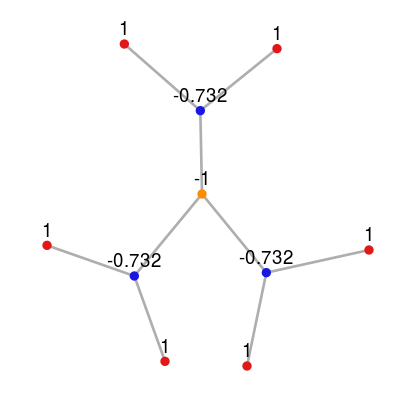}
\includegraphics[scale=0.5]{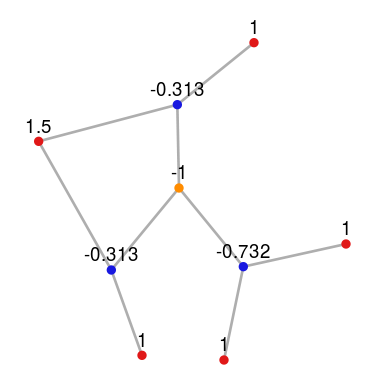}
\includegraphics[scale=0.5]{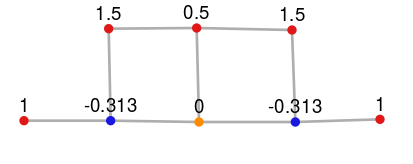}
\includegraphics[scale=0.5]{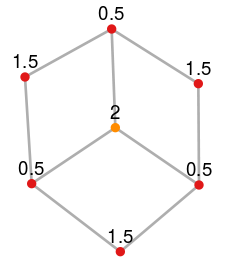}
\caption{The 2-balls $D1$, $D2$, $D3$ and $D4$}
\end{figure}
\begin{figure}
\centering
\includegraphics[scale=0.5]{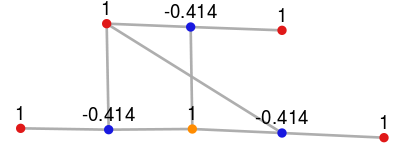}
\includegraphics[scale=0.5]{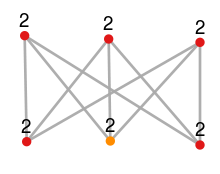}
\includegraphics[scale=0.5]{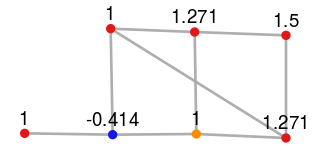}
\caption{The 2-balls $D5$, $D6$ and $D7$}\label{fig_D7}
\end{figure}

\begin{theorem}
Let $G=(V,E)$ be a simple 3-regular graph. Then $G$ satisfies $CD(0,\infty)$ if and only if it is a prism graph $Y_n$ for some $n\geq 3$ or a  M\"obius ladder $M_{k}$ for some $k\geq 2$.
\end{theorem}
\begin{proof}
By looking at the curvature of the possible 2-balls of cubic graphs we see that only $A1, B1, B2, C1, C3, C4, D3, D4, D5, D6$ and $D7$ can appear in a non-negatively curved graph, since in $C2, C5, D1$ and $D2$ the curvature at the centre is negative. Furthermore note that $A1$ and $D6$ are already cubic graphs, namely $K_{4}=M_2$ and $K_{3,3}=M_3.$ Thus $A1$ and $D6$ do not appear locally anywhere else.

Note that none of the 2-balls $B1, B2, C1, C3, C4, D3, D4, D5$ or $D7$ contains a bridge, that is, an edge such that removing it would disconnect the graph, connected to the centre vertex.  Thus no non-negatively curved cubic graph contains a 2-ball with a bridge connected to its centre (see \cite[Theorem 6.4]{CLP2016} for a more general version of this fact). Now, if we try to extend any of the 2-balls $B2$, $C1$, $C3$, or $D7$ to a cubic graph, this forces a bridge in a 2-ball of one of the vertices at distance two from the centre. Also, extending $B1$ forces a bridge either at a vertex at distance two or three from the centre. Thus $B1$, $B2$, $C1$, $C3$ and $D7$ cannot exist as the 2-ball in a non-negatively curved cubic graph. The 2-ball $D5$ can be extended to a cubic graph without bridges only by joining together all the three vertices with degree one, but this creates a graph with negative curvature at most of the vertices.

We now only have the 2-balls $C4, D3$ and $D4.$ Considering one of the vertices in the 1-sphere of $C4$, it is clear that the only way to extend this into the centre of one of $C4, D3$ or $D4$ is to extend it into $C4$, giving the triangular prism $Y_{3}.$ A similar argument shows that the only non-negatively curved cubic graph containing $D4$ is the 3-dimensional cube, i.e $Y_{4}.$

This leaves us only to classify graphs that have $D3$ as a 2-ball everywhere. It is clear that there are infinitely many of these graphs, i.e. the M\"obius ladders and the prism graphs.
\end{proof}

Remark, that of these graphs only the smallest ones, namely $K_4=M_2$, $K_{3,3}=M_3$, the 3-prism $Y_3$ and the cube $Y_4$,  have positive curvature at all vertices.

\section{Ollivier-Ricci classification}
Let us first see how under Olliver-Ricci curvature $\kappa_0$ non-negatively curved graphs look locally. The following lemma shows that in order to have non-negative Olliver-Ricci curvature on an edge $xy$, the edge must be on a triangle or a square.

\begin{lemma}\label{local_structure} Let $G=(V,E)$ be a 3-regular graph, and let $C_n$ be the smallest cycle $C_n$ supporting the edge $xy\in E(G)$. Then we have the following cases:
\begin{enumerate}
\item[i)] If $n=3$, then $\kappa_0(x,y)\geq 1/3$;
\item[ii)] If $n=4$, then $\kappa_0(x,y)=0$;
\item[iii)] If $n\geq 5$, then $\kappa_0(x,y)<0$.
\end{enumerate}
\end{lemma}

\begin{proof}
Assume that $n=3$, that is, the edge $xy\in E(G)$ is on a triangle. Then the Wasserstein distance $W_1(\mu_x,\mu_y)\leq 2\cdot \frac{1}{3}=\frac{2}{3}$, since the mass distribution at the common neighbour of $x$ and $y$ does not need to me moved by the transport plan. Thus $\kappa_0(x,y)\geq 1/3$.

Assume then that $n=4$, and thus that the edge $xy\in E(G)$ is  on a square but not on a triangle. Then there exists a perfect matching between the 1-spheres $S_1(x):=\{z\in G\; | \; d(x,z)=1\}$ and $S_1(y):=\{z\in G\; | \; d(y,z)=1\}$: if the neighbours of $x$ and $y$ that lie on a square are denoted $x_1$ and $y_1$ and the other two neighbours $x_2$ and $y_2$, then choose to the matching the edges $x_1y_1$, $xx_2$ and $yy_2$. The transport plan that moves masses along this perfect matching is optimal regardless whether  the vertices $x_2$ and $y_2$ are adjacent. Thus we have $W_1(\mu_x,\mu_y)=3\cdot \frac{1}{3}=1$ and $\kappa_0(x,y)=0$.

Assume then that $xy$ is not on a triangle or a square. We can then define a 1-Lipschitz function $\phi$ as follows:
  \begin{equation}\label{eqpot}
    \phi(u) =
    \begin{cases}
      2, & \text{if } u \sim x \text{ and } u \neq y \\
      1, & \text{if } u = x \text{ or } u=y \\
      1, & \text{if } u \sim v \text{ for some } v \sim x, v \neq y \\
      0, & \text{otherwise}
    \end{cases}
  \end{equation}
For this function $\sum_{u \in V}\phi(x)(\mu_x^0(u)-\mu_y^0(u)) = 2(\frac{1}{3}-0) + 2(\frac{1}{3}-0) + 1(0-\frac{1}{3}) + 1(\frac{1}{3} - 0) = \frac{4}{3}$. Thus $W_1(\mu_x^0,\mu_y^0) \geq \frac{4}{3}$, and by the Kantorovich duality \ref{Kantorovich} we have $\kappa_0(x,y) \leq -\frac{1}{3} < 0$.  
\end{proof}

Let us now consider the graphs with $\kappa_0\geq 0$ on all edges. We divide the classification into two parts, considering graphs with constant curvature $\kappa_0=0$ in Theorem \ref{flat_g4} and then graphs with positive curvature on at least one edge in Theorem \ref{flat_g3}.

\begin{theorem}\label{flat_g4} Let $G=(V,E)$ be a simple 3-regular graph. Then  $\kappa_0(x,y)=0$ for all $xy\in E(G)$ if and only if $G$ is a prism graph $Y_n$ for some $n\geq 4$ or a  M\"obius ladder $M_{k}$ for some $k\geq 3$.
\end{theorem}

\begin{proof}
Since the prism graphs $Y_n$, $n\ge 4$,  and a M\"obius ladders $M_{k}$, $k\geq 3$, are triangle free and every edge of them lies on a square, they by Lemma \ref{local_structure} have $\kappa_0(x,y)=0$ for all edges $xy$.

Assume then that $\kappa_0(x,y)=0$ for all $xy\in E(G)$. From lemma \ref{local_structure} we know that all edges of $G$ are on a $C_4$ and that there are no triangles in the graph.

Consider an edge $xy\in E(G)$. Since the graph is triangle free, the 1-spheres $S_1(x)$  and $S_1(y)$ are disjoint. Denote the vertices as $S_1(x)=\{y,x_1,x_2\}$, $S_1(y)=\{x,y_1,y_2\}$. Since $xy$ lies on a square, we can assume that $x_1\sim y_1$. Let us show that there always exists a 3-ladder $L_3$ (see Figure \ref{lnnames} for $n$-ladder $L_n$) as a subgraph in $G$: Since $\kappa_0=0$ on all edges, also the edges $xx_2$ and $yy_2$ must lie on squares. This is obtained either if $x_2\sim y_2$, or if $x_2\sim y_1$ and $y_2\sim x_1$, and in both cases we have a subgraph $L_3$ in the graph.  

\begin{figure}[h!]
  \centering
\begin{tikzpicture}
  \tikzset{vertex/.style={circle, draw, fill=black!50,
                        inner sep=0pt, minimum width=4pt}}
  \tikzset{emptyvertex/.style={shape=circle,minimum size=0.6cm}}

  \node (a1) [vertex, label=above:$a_1$]{} ;
  \node (b1) [vertex, below=of a1, label=below:$b_1$] {};

  \node (a2) [vertex, right=of a1, label=above:$a_2$] {};
  \node (b2) [vertex, right=of b1, label=below:$b_2$] {};

  \node (a3) [vertex, right=of a2, label=above:$a_3$] {};
  \node (b3) [vertex, right=of b2, label=below:$b_3$] {};

   \node (a4) [emptyvertex, right=of a3] {};
  \node (b4) [emptyvertex, right=of b3] {};

  \node (an1) [vertex, right=of a4, label=above:$a_{n-1}$] {};
  \node (bn1) [vertex, right=of b4, label=below:$b_{n-1}$] {};

  \node (an) [vertex, right=of an1, label=above:$a_n$] {};
  \node (bn) [vertex, right=of bn1, label=below:$b_n$] {};
  \path
  (a1) edge (b1)
  (a1) edge (a2)
  (b1) edge (b2)
  (a2) edge (b2)
  (a2) edge (a3)
  (b2) edge (b3)
  (a3) edge (b3)
  (a3) edge[dashed] (a4)
  (b3) edge[dashed] (b4)
  (a4) edge[dashed] (an1)
  (b4) edge[dashed] (bn1)
  (an1) edge (bn1)
  (an1) edge (an)
  (bn1) edge (bn)
  (an) edge (bn)
  ;

\end{tikzpicture}
  \caption{The ladder graph $L_n$ with labelling}\label{lnnames}
\end{figure}
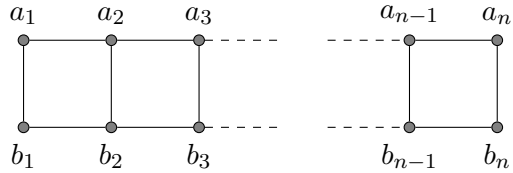

Assume now that $L_n\subset G$ for some $n\ge 3$, and denote the vertices as in Figure \ref{lnnames}. The remaining edges from $a_n$ and $b_n$ also have to lie on squares. If there are no other vertices in the graph, that can happen two ways: either $a_n\sim a_1$ and $b_n\sim b_1$ or $a_1\sim b_n$ and $b_1\sim a_n$. If there is new vertex $a_{n+1}$ such that $a_n\sim a_{n+1}$, then there must also be a second new vertex $b_{n+1}$ with $b_n\sim b_{n+1}$ such that $a_{n+1}\sim b_{n+1}$. But this creates a $L_{n+1}$. Thus one of the following holds:
\begin{enumerate}
\item[i)] $a_1\sim a_n$ and $b_1\sim b_n$,
\item[ii)] $a_1\sim b_n$ and $b_1\sim a_n$, or
\item[iii)] $L_{n+1}\subset G$.
\end{enumerate}
In the first case $G$ is the prism graph $Y_n$, in the second case $G$ is the  M\"obius ladder $M_{n}$.
Remark that the first case is only possible if  $n\ge 4$, otherwise there  would be a triangle $a_1 a_2 a_3$. In the third case we can continue by induction, obtaining larger prism graphs and M\"obius ladders. 
\end{proof}

\begin{theorem}\label{flat_g3}
The only 3-regular graphs with $\kappa_0\geq 0$ on all edges and $\kappa_0(x,y)>0$ on at least one edge are the complete graph $K_4$ and the prism graph $Y_3$.
\end{theorem}

\begin{proof}
By lemma \ref{local_structure} we can assume, that the girth $g(G)=3$. Let us construct all possible cubic graphs starting from a triangle using the knowledge that all edges must lie on a triangle or on a square. So, let $xy$ be an edge that lies on a triangle $xyz$. Denote the third neighbour of $x$ by $x_1$. Since $G$ is $3$-regular, either $y\sim x_1$ or there is another vertex $y_1\sim y$. In the former case the only vertices  with degree less than three are $z$ and $x_1$, and so in order to the last edge from $z$ to be an a triangle or a square, we must have $z\sim x_1$. This gives the graph $K_4$, which is the same as the smallest  M\"obius ladder $M_2$.

Consider then the latter case $y_1\sim y$. Then there exists two isomorphically different possibilities: either $z$ is adjacent to $x_1$ (or, isomorphically, to $y_1$) or to a new vertex $z_1$. In the former case the last edge from $x_1$ has to go to $y_1$ in order to be on a square or a triangle. But that leaves only the vertex $y_1$ with degree less than 3, and the construction cannot be continued. In the latter case when $z$ is adjacent to a new vertex $z_1$, the other two edges from $z_1$ has to go to $x_1$ and $y_1$ to be on squares. That leaves only the vertices $x_1$ and $y_1$ with degrees less than three. Since $d(x_1,y_1)=2$, they must either be adjacent, which gives $Y_3$, or to have another common neighbour, say $u$. But then $u$ would be the only vertex with degree less than three, and the construction could not be continued. Thus, the only possible graphs are $K_4=M_2$ and $Y_3$.
\end{proof}
Remark, that $M_2$ is the only cubic graph that has positive Ollivier-Ricci curvature on all edges, namely $\kappa_0= 2/3$, since for $Y_3$ the curvature on some edges is zero.

Combining these two theorems we have the following classification result:

\begin{corollary}
Let $G=(V,E)$ be a simple 3-regular graph. Then $\kappa_0(x,y)\ge0$ for all $xy\in E(G)$ if and only if $G$ is a prism graph $Y_n$ for some $n\geq 3$ or a  M\"obius ladder $M_{k}$ for some $k\geq 2$.
\end{corollary}

\section{Final comments}
For a finite graph $G = (V,E)$ let $\lambda_{1}(G)$ denote the smallest non-zero eigenvalue of its Laplacian.
\begin{definition}
Let  $d\in\mathbb{N}.$ Let $(G_{i})_{i\in\mathbb{N}}$ be an infinite family of $d$-regular graphs such that $|V_{n}|\rightarrow\infty.$ We say that $(G_{i})_{i\in\mathbb{N}}$ is a family of expanders if there exists an $\varepsilon>0$ such that
$$\lambda_{1}(G_{i})\geq \varepsilon$$
for all $i\in\mathbb{N}.$
\end{definition}

It is an open question on whether the space of non-negatively curved graphs, in both the Bakry-\'Emery and Ollivier-Ricci sense, contains expanders. Our classification shows that no cubic expanders with exist.

\begin{theorem}
There is no family of 3-regular expanders that is non-negatively curved in Bakry-\'Emery or Ollivier-Ricci sense.
\end{theorem}

\begin{proof}
By our main Theorem \ref{main_thm} the only 3-regular non-negatively curved graphs under either curvature notion are the prism graphs $Y_n$ and the M\"obius ladders $M_n$. We show that these graphs do not have a spectral gap, that is, the second smallest eigenvalue converges to zero, when $n\rightarrow \infty$, and thus they cannot form a family of expanders.

The prism graphs $Y_n$ are Cartesian products of a cycle $C_n$ and an edge $K_2$. Thus their Laplacian eigenvalues are sums of the eigenvalues of  $C_n$ and $K_2$, see e.g. \cite{BH12}. The Laplacian spectrum for $K_2$ is $\{0,2\}$ and for $C_n$ $\{2-2\cos(\frac{2\pi j}n)\}$, where $j=0,\ldots n-1$. Therefore the smallest non-zero eigenvalue of a prism graph $Y_n$ is $\lambda_1(Y_n)=2-2\cos(\frac{2\pi}{n})$, and thus $\lambda_1(Y_n)\rightarrow 0$, when $n\rightarrow \infty$.

The Laplacian eigenvalues of the M\"obius ladders can be calculated by considering them as cycles $C_{2n}$ with opposite vertices attached. Then it is easy to see that the Laplacian is a circulant matrix, where the first column is $v_0=3$, $v_1=-1$, $v_2=\ldots =v_{n-1}=0$, $v_n=-1$, $v_{n+1}=\ldots = v_{2n-2}=0$, $v_{2n-1}=-1$, and the remaining columns are cyclic permutations of the first one with offset equal to the column index. The Laplacian eigenvalues  of $2n\times 2n$ circulant matrices are
$$\{v_0+v_{2n-1}\omega_j+v_{2n-2}\omega_j^2+\ldots +v_{1}\omega_j^{2n-1}\},$$
where $v$ is the first column, $\omega_j=\exp(\frac{i 2\pi j}{2n})$ are the $j$th roots of unity, and $j=0,\ldots,2n-1$ (see \cite{D94}). Thus the Laplacian spectrum of the M\"obius ladder $M_n$ is
$\{3+(-1)^{j+1}-2\cos(\frac{\pi j}{n})\},$
where $j=0,\ldots,2n-1$. The smallest non-zero eigenvalue is
$\lambda_1(M_n)=3+(-1)^3-2\cos(\frac{\pi 2}{n})$, and thus $\lambda_1(M_n)\rightarrow 0$,
when $n\rightarrow \infty$.
\end{proof}

\begin{rem}
This result can also be obtained by showing that the prism graphs and the M\"obius ladders are abelian Cayley graphs, and then applying the result in \cite{AR94} that Abelian Cayley graphs do not contain expanders.
\end{rem}

{\bf Acknowledgements}
\\
The authors are grateful to Prof. Norbert Peyerimhoff for his continued support and guidance. DC and SL would like to acknowledge that this work was supported by the EPSRC Grant EP/K016687/1 ``Topology, Geometry and Laplacians of Simplicial Complexes''. DC would like to thank Aalto University and RK and SL Durham University for their hospitality during research visits, where some of the above work was carried out.  DC wants to thank the EPSRC for financial support through his postdoctoral prize.

\end{document}